\documentclass[a4paper,12pt]{amsart}
\usepackage{amssymb,amsthm,fullpage,amsmath}
\usepackage{ulem}
\newcommand{\R}{\mathbb{R}}

\newcommand{\N}{\mathbb{N}}

\renewcommand{\S}{\mathbb{S}}
\newtheorem{theorem}{Theorem} 
\newtheorem{lemma}{Lemma} 
\newtheorem{proposition}{Proposition} 
\theoremstyle{definition}
\newtheorem{definition}{Definition} 
\newtheorem{remark}{Remark}
\def\intj{\mbox{int}}
\renewcommand\div{\mbox{div}\,}
\def\curl{\text{curl}\,}
\def\eps{\varepsilon}
\begin{document}
\title{Prescribing the motion of a set of particles \\ in a $3$D perfect fluid}
%
%
\author{O. Glass}
\address{Ceremade, Universit\'e Paris-Dauphine, Place du Mar\'echal de Lattre de Tassigny, 75775 Paris Cedex 16, France}
\curraddr{}
\email{glass@ceremade.dauphine.fr}
\thanks{}

\author{T. Horsin}
\address{Universit\'e de Versailles, 45 avenue des Etats-Unis, F78030 Versailles, France}
\curraddr{}
\email{horsin@math.uvsq.fr}
\thanks{}
\subjclass[2000]{Primary }
\date{17/08/2011}
\dedicatory{}
\begin{abstract}
We establish a result concerning the so-called Lagrangian controllability of the Euler equation for incompressible perfect fluids in dimension $3$. More precisely we consider a connected bounded domain of $\R^{3}$ and two smooth contractible sets of fluid particles, surrounding the same volume. We prove that given any initial velocity field, one can find a boundary control and a time interval such that the corresponding solution of the Euler equation makes the first of the two sets approximately reach the second one. 
\end{abstract}
\maketitle
%
%
%
%
%
%
%
%
%
\section{Introduction}
\subsection{Presentation of the problem}
In this paper, we are concerned with the Lagrangian controllability of the three dimensional Euler equation for perfect incompressible fluids by means of a boundary control. The problem under view is the following. \par
Let $\Omega$ be a smooth bounded domain of ${\mathbb R}^3$ and let $\Gamma$ be a nonempty open part of its boundary $\partial \Omega$. Given $T>0$, we consider the classical Euler system for perfect incompressible fluids in $\Omega$:
\begin{align}
\label{eq:euler.1}
&\partial_t u+(u \cdot\nabla) u + \nabla p=0 \ \text{ in } \ (0,T)\times \Omega, \\ 
\label{eq:euler.2}
&\div u=0 \ \text{ in } \ (0,T)\times \Omega, \\
\label{eq:euler.3}
&u_{|t=0}=u_0 \ \text{ in }\ \Omega.
\end{align}
As boundary conditions, one usually considers the impermeability condition on $\partial \Omega$:
\begin{equation*}
u \cdot n=0 \ \text{ on } \ (0,T) \times \partial \Omega,
\end{equation*}
where $n$ stands for the unit outward normal vector field on $\partial \Omega$.
However, in the problem under view, the impermeability condition is merely imposed on $\partial \Omega \setminus \Gamma$:
\begin{equation}
\label{eq:euler.4}
u \cdot n=0 \ \text{ on } \ (0,T)\times (\partial \Omega\setminus \Gamma).
\end{equation}
Consequently, since we do not prescribe boundary values for $u$ on $(0,T)\times \Gamma$, the problem stated as above is underdetermined. It is in fact an implicit control problem; in other words, we consider the boundary values on $\Gamma$ as a control on this system, i.e. a way to influence the fluid in a prescribed way. If one wants to close the system, as shown by Khazhikov \cite{Ka}, one may consider as boundary conditions on $\Gamma$ the following ones:
\begin{equation*}
u \cdot n\text{ on }(0,T)\times \Gamma \ \text{ and } \ 
\curl u \wedge n\text{ on } \{ (t,x) \in (0,T)\times \Gamma\text{ such that } u \cdot n<0 \},
\end{equation*}
that is, the normal part of the velocity field on $\Gamma$ and the tangential entering vorticity. We refer to \cite{Ka} for more details on this notion of boundary values. Due to the complexity of this formulation of boundary values, and as in the literature concerning the controllability of the Euler equation (see e.g. \cite{93COR}, \cite{96COR-2}, \cite{00Gla}, \cite{glassa}), we prefer not to refer to them explicitely and to look for the solution $(u,p)$ itself. Once a solution of \eqref{eq:euler.1}-\eqref{eq:euler.4} is given, its suitable trace on $(0,T)\times \Gamma$ can be thought as the control. \par
\ \par
The question that we raise is the possibility of prescribing the motion of a set of particles driven by a fluid governed by \eqref{eq:euler.1}-\eqref{eq:euler.4}. It is classical in fluid mechanics that one can describe the motion of a fluid from two different points of view. One possibility, referred as the Eulerian description of the fluid, is to describe the velocity of fluid particles that pass ``through'' a given point. The other possibility, referred as the Lagrangian description of the fluid consists in following fluid particles along the flow. To that purpose one has to solve the differential equation associated to the velocity field. \par
The notion of controllability that we consider in this paper is different from the usual notion of controllability and, in some sense, closer to the Lagrangian description of the fluid than the usual notion of controllability. The standard sense of controllability would refer to the possibility of driving the velocity field (the state of the system), from one prescribed value to a fixed target. The study of the controllability of the Euler equation has been initiated by J.-M. Coron in \cite{93COR} and \cite{96COR-2}, and then studied by the first author in \cite{00Gla} and \cite{glassa}. Also related to the boundary control of the velocity field, the question of asymptotic stabilization of this system around zero was studied by J.-M. Coron in $2$-D simply connected domains (see \cite{CORSTA99}) and by the first author for what concerns more general $2$-D domains (see \cite{Glass-stabilisation}). \par
In this paper, we consider the problem of prescribing the displacement of a set of particles, rather than the velocity field in final time. In \cite{GlHo08}, the authors showed that in $2$-D, one can indeed prescribe approximately the motion of some specific sets of fluids. The goal of the present paper is to consider the case of the dimension $3$.  \par
\subsection{Notations and definitions}
In this section, we fix the notations and give the basic definitions. \par
In the sequel, smooth curves, surfaces or maps will mean $C^\infty$ ones. We will denote by $C^\omega$ the class of real analytic curves/surfaces. The volume of a borelian set $A$ in $\R^{3}$ will be denoted by $|A|$. 
Given a suitably regular vector field $u$, we will denote by $\phi^u$ the flow of $u$, defined (when possible) by 
\begin{equation} \label{eq:flotdeu}
\partial_t\phi^u(t,s,x)=u(t,\phi^u(t,s,x)) \ \text{ and } \ u(s,s,x)=x.
\end{equation}
For $A$ a subset of $\R^{3}$ and $\eta>0$, we will denote:
\begin{equation*}
V_{\eta}(A) := \{ x \in \R^{3} \ / \ d(x,A) < \eta \}.
\end{equation*}
We will also use the notation $L^{p}(0,T;X(\Omega(t)))$ or $C^{k}([0,T];X(\Omega(t)))$ of time-dependent functions with values in a space $X$ of functions on a domain depending on time $\Omega(t) \subset \R^{3}$. This refers to the space functions that can be extended to $L^{p}(0,T;X(\R^{3}))$ or $C^{k}([0,T];X(\R^{3}))$. \par
\ \par
We now recall some definitions from differential geometry (see e.g. \cite{Hirsch}).
\begin{definition}
A Jordan surface $\gamma$ in $\R^{3}$ is the image of $\S^2$ by some homeomorphism $h$. According to the Jordan-Brouwer Theorem (see \cite{Brou12} or \cite{GP}), $\R^{3} \setminus \gamma$ has two connected components, one of which is bounded and will be denoted by $\intj(\gamma)$. Moreover when the homeomorphism $h$ is a diffeomorphism, we have a unit outward normal vector field on $\gamma$, which we will denote by $\nu$.
\end{definition}
\begin{definition}
Two Jordan surfaces $\gamma_0$ and $\gamma_1$ embedded in $\R^3$ are said to be isotopic in $\Omega$, if there exists a continuous map $\mathfrak{I}:[0,1]\times \S^2\to \Omega$ such that $\mathfrak{I}(0)=\gamma_{0}$, $\mathfrak{I}(1)=\gamma_{1}$ and for each $t\in [0,1],\,\mathfrak{I}(t,\cdot)$ is an homeomorphism of $\S^2$ into its image. When, for some $k\in \N \setminus \{0\}$, this homeomorphism is a $C^{k}$-diffeomorphism with respect the space variable, $\mathfrak{I}$ will be said to be a $C^k$-isotopy, or, when $k=\infty$, a smooth isotopy. A one-parameter continuous family of diffeomorphisms of $\Omega$ is called a diffeotopy of $\Omega$.
\end{definition}
When their regularity is not specified, the geometrical objects are considered to be smooth in the sequel. 
We are now in position to define the Lagrangian controllability (the corresponding two-dimensional notion was introduced in \cite{GlHo08}). 
\begin{definition}\label{Def:ELC}
Let $T>0$ be a given time. We will say that the exact lagrangian controllability of \eqref{eq:euler.1}-\eqref{eq:euler.4} holds if
given two Jordan surfaces $\gamma_0$ and $\gamma_1$ included in $\Omega$ such that
\begin{gather} \label{CondGamma1}
\gamma_{0} \text{ and } \gamma_{1} \text{ are isotopic in } \Omega, \\
\label{CondGamma2}
|\intj(\gamma_{0})| = |\intj(\gamma_{1})|,
\end{gather}
and given a regular initial data $u_{0}$ satisfying
\begin{gather} \label{CIDiv}
\div u_{0}=0 \ \text{ in } \ \Omega, \\
\label{CIBord}
u_{0} \cdot n =0 \ \text{ on } \ \partial \Omega \setminus \Gamma,
\end{gather}
there exists a solution $(u,p)$ of \eqref{eq:euler.1}-\eqref{eq:euler.4} such that one has
\begin{gather}
\label{eq:exact2}
\forall t\in [0,T],\ \phi^u(t,0,\gamma_0)\subset \Omega, \\
\label{eq:exact1}
\phi^u(T,0,\gamma_0)=\gamma_1.
\end{gather}
up to reparameterization.
\end{definition}
Let us say a few words concerning the regularity of $u$. In general, strong solutions of the Euler equation are considered in a H\"older or Sobolev space with respect to $x$, which is included in the space of Lipschitz functions. In the sequel, the solutions will be taken in the H\"older space $C^{k,\alpha}(\Omega)$, $k \geq 1$. \par
If the exact Lagrangian controllability does not occur, one may try to weaken this definition as follows.
\begin{definition}\label{Def:ALC}
We will say that the approximate Lagrangian controllability of \eqref{eq:euler.1}-\eqref{eq:euler.4} holds in time $T$ and in norm $C^k$ if given $\gamma_{0}$, $\gamma_{1}$ and $u_{0}$ as above, and given any $\eps>0$, there exists a solution $(u,p)$ of  \eqref{eq:euler.1}-\eqref{eq:euler.4} such that \eqref{eq:exact2} holds and that
\begin{equation}
\label{eq:approx}
\| \phi^u(T,0,\gamma_0)-\gamma_1 \|_{C^k(\S^{2})}<\eps,
\end{equation}
up to reparameterization.
\end{definition}
\begin{remark} \label{rem.1}
It is easy to see that the conditions that
$\gamma_0$ and $\gamma_1$ should be isotopic and enclose the same volume are necessary in order that a smooth volume-preserving flow drives $\gamma_{0}$ to $\gamma_{1}$. In particular $|\intj(\gamma_0)|=|\intj(\gamma_1)|$ comes from the incompressibility of the fluid.
\end{remark}
\begin{remark}
Condition \eqref{eq:exact2} allows to make sure that one ``controls'' the fluid zone for all time by making it stay in the domain; Conditions \eqref{eq:exact1} or \eqref{eq:approx} would not have a clear meaning without it.
\end{remark}
\subsection{Main result}
The main result of this paper is the following, establishing an approximate Lagrangian controllability property.
\begin{theorem}\label{th:fond.1}
Let $\alpha \in (0,1)$ and $k \in \N \setminus \{0\}$. Consider $u_0\in C^{k,\alpha}({\Omega};\R^{3})$ satisfying \eqref{CIDiv}-\eqref{CIBord} and $\gamma_0$ and $\gamma_1$ two contractible $C^\infty$ embeddings of $\mathbb S^2$ in $\Omega$ satisfying \eqref{CondGamma1}-\eqref{CondGamma2}. Then for any $\eps>0$, there exist a time $T>0$ and a solution $(u,p)$ in $L^{\infty}(0,T;C^{k,\alpha}(\Omega;\R^{4}))$ to \eqref{eq:euler.1}-\eqref{eq:euler.4} on $[0,T]$ such that \eqref{eq:exact2} and \eqref{eq:approx} hold (up to reparameterization).
%
%
\end{theorem}
%
%
%
%
\begin{remark}
Let us check that, as in the two-dimensional case (see \cite{GlHo08}), the exact lagragian controllability does not hold. As is classical, in $3$-D the vorticity of the fluid
\begin{equation} \label{DefVor}
\omega:=\curl u,
\end{equation}
satisfies the equation
\begin{equation} \label{EqVor}
\partial_t\omega +(u \cdot \nabla) \omega=(\omega \cdot \nabla) u.
\end{equation}
Denoting $w(t,x)=\omega(t,\phi^u(t,0,x))$, we see that
\begin{equation*}
\partial_t w=(w \cdot \nabla)u(t,\phi^u(t,0,x)). 
\end{equation*}
Thus if initially $\curl u_0=0$ in a neighborhood of $\gamma_0$ then $\curl u=0$ in a neighborhood of $\phi^u(t,0,\gamma_0)$ as long as \eqref{eq:exact2} is true. \par
Now, since $u$ satisfies \eqref{eq:euler.2}, for each time $t$, in such a neighborhood of $\phi^u(t,0,\gamma_0)$ and away from $\partial \Omega$, $u(t,\cdot)$ is locally the gradient of a harmonic function, and hence is real analytic . Therefore, if $\gamma_0$ is a real-analytic submanifold, so will be $\phi^u(t,0,\gamma_0)$ for each $t\in [0,T]$. It follows that the exact Lagrangian controllability does not hold in general, since one may take $\gamma_1$ smooth but non analytic.
\end{remark}
\begin{remark}
As will be clear from the proof, the time $T$ whose existence is granted by Theorem \ref{th:fond.1} can be made arbitrarily {\it small}. The result does not require a time long enough in order for the controllability property to take place, but a time small enough to make sure that no blow up phenomenon occurs. 
\end{remark}
Let us now briefly describe the strategy to prove Theorem \ref{th:fond.1}. First, we use a construction due to A. B. Krygin \cite{krygin} to obtain a solenoidal vector field $X\in C^\infty_0((0,1) \times \Omega;\R^{3})$, such that 
\begin{gather*}
\forall t\in [0,1], \ \ \phi^X(t,0,\gamma_0)\subset \Omega, \\
\phi^X(1,0,\gamma_0)= \gamma_{1}.
\end{gather*}
Then we will prove that there exist potential flows (that is, time-dependent gradients of harmonic functions in $\Omega$), which approximate suitably the action of $X$ on $\gamma_{0}$. This gives us a solution $(\bar{u},\bar{p})$ of \eqref{eq:euler.1}-\eqref{eq:euler.4} with $u_0=0$ and $T=1$ such that 
\begin{equation*}                                                                                                                                                                                                                                                                                                                                                            \| \phi^{\bar{u}}(0,t,\gamma_0)-\phi^X(0,t,\gamma_0)\|_{C^k(\S^2)}\leq \eps.                                                                                                                                                                                                                                                                                                                                                                                                    
\end{equation*}
Then we will use a time-rescaling argument due to J.-M. Coron \cite{96COR-2} (see also \cite{00Gla}), to get the result for non trivial $u_{0}$ and $T$ small enough. \par
\ \par
The rest of the paper is divided as follows. In Section \ref{Sec:PotentialFlows}, we consider the existence of the above mentionned vector field $X$, and then of the potential flows. In Section \ref{Sec:Proof}, we prove Theorem \ref{th:fond.1} by constructing the solution $(u,p)$ taking the initial condition $u_{0}$ into account. \par
\ \\
{\bf Acknowledgements.} The authors are partially supported by the Agence Nationale de la Recherche (ANR-09-BLAN-0213-02). They thank Jean-Pierre Puel for many useful discussions leading to this paper.
%
%
%
%
%
%
%
%
%
%
\section{Potential flows}
\label{Sec:PotentialFlows}
%
%
%
%
%
\subsection{A solenoidal vector field mapping $\gamma_0$ onto $\gamma_1$}
\label{subsec:proof1}
%
%
%
The vector field $X$ is obtained as a direct consequence of the following result due to A. B. Krygin.
\begin{theorem}[\cite{krygin}] \label{th:isotop.1}
If $\gamma_0$ and $\gamma_1$ are as in Theorem $1$, then there exists a volume-preserving diffeotopy $h \in C^\infty([0,1] \times \Omega;\Omega)$ such that $\partial_{t} h$ is compactly supported in $(0,1) \times \Omega$, $h(0,\gamma_{0})=\gamma_0$ and $h(1,\gamma_{0})=\gamma_1$.
\end{theorem}
A direct consequence is that the smooth vector field
\begin{equation} \label{DefX}
X(t,x):= \partial_{t} h (t,h^{-1}(x)),
\end{equation}
is compactly supported in $(0,1) \times \Omega$ and satisfies
\begin{equation*}
\phi^{X}(1,0,\gamma_{0}) = \gamma_{1},
\end{equation*}
and 
\begin{equation*}
\div X =0 \ \text{ in } \ (0,1) \times \Omega.
\end{equation*}
\subsection{Moving the fluid zone by potential flows}
In this subsection we prove the following.
\begin{proposition} \label{th:3}
Let $\gamma_0$ be a $C^\infty$ contractible two-sphere embedded in $\Omega$ and consider $X\in C^0([0,1]; C^{\infty}(\overline{\Omega};\R^3))$ a solenoidal vector field such that 
\begin{equation}
\forall t\in [0,1],\  \phi^X(t,0,\gamma_0)\subset \Omega,
\end{equation}
and let 
\begin{equation*}
\gamma_1=\phi^X(1,0,\gamma_0).
\end{equation*}
For any $\eps>0$ and $k\in \N$ there exists $\theta\in C^{\infty}_0((0,1) \times \overline{\Omega};\R)$ such that 
\begin{eqnarray}
\label{eq:thetaharm}
\forall t\in [0,1],\ \Delta_x \theta=0 \ \text{ in } \ \Omega, \\
\label{eq:4.9}
\dfrac{\partial \theta}{\partial n}=0 \ \text{ on } \ [0,1] \times (\partial \Omega\setminus\Gamma), \\
\label{eq:thetarestedansomega}
\forall t\in [0,1], \ \phi^{\nabla \theta}(t,0,\gamma_0) \subset \Omega, \\
\label{eq:thetafaitletravail}
\| \phi^{\nabla\theta}(1,0,\gamma_0)-\gamma_1\|_{C^k(\S^2)} \leq \eps, 
\end{eqnarray}
up to reparameterization.
\end{proposition}
As in \cite{GlHo08} we prove this proposition by assuming first that $\gamma_0$ is an embedded real analytic $2$-sphere and that $X$ is real analytic in $x$.
Then we progressively reduce the assumptions to the framework of Proposition \ref{th:3}.
\subsubsection{Case of an analytic $2$-sphere moved by an analytic isotopy}
The goal of this paragraph is to prove the following proposition.
\begin{proposition}\label{ProStep1}
The conclusions of Proposition \ref{th:3} are satisfied if we assume that $\gamma_0$ is an analytic $2$-sphere ({\it i.e.} $\gamma_{0}$ is the image of $\S^{2}$ by a real-analytic embedding $f_0$) and that $X$ is moreover in $C^{0}([0,1]; C^\omega(\Omega;\R^3))$.
\end{proposition}
It is clear that in that case, $\phi^X$ is volume preserving real analytic isotopy between $\gamma_0$ and $\gamma_1$. 
The first step to establish Proposition \ref{ProStep1} is the following.
\begin{lemma}
\label{prop:etendrepsi}
Let $t\mapsto \gamma(t)$ be a $C^{0}([0,1];C^\omega(\S^2))$ family of contractible $2$-spheres in $\Omega$ and $X\in C([0,1];C^\omega(\overline{\Omega}))$ be a family of vector field such that 
\begin{equation} \label{CondXNu}
\int_{\gamma(t)}X \cdot \nu \, d\sigma=0,
\end{equation}
then there exists  $\eta>0$ and $\psi\in C^{0}([0,1];C^\infty(V_\eta[\intj(\gamma(t))];\R))$ such that
\begin{eqnarray}
\label{eq:psiharm}
\forall t\in [0,1], \ \Delta_x\psi=0 \ \text{ in } \ {V}_\eta[\intj(\gamma(t))], \\
\label{eq:dernormpsiegalX}
\forall t\in [0,1], \ \frac{\partial \psi}{\partial \nu}=X \cdot \nu \ \text{ on } \ \gamma(t).
\end{eqnarray}
\end{lemma}
In other words, Lemma \ref{prop:etendrepsi} expresses that the solution of the Neumann system
\begin{equation} \label{SysNeumann}
\left\{ \begin{array}{l}
 \Delta_x\psi=0 \ \text{ in } \ \intj(\gamma(t)), \\
 \frac{\partial \psi}{\partial \nu}=X \cdot \nu \ \text{ on } \ \gamma(t),
\end{array} \right.
\end{equation}
can be extended across the boundary $\gamma(t)$ uniformly in $t$. \par
\begin{proof}
An equivalent lemma in dimension $2$ is given in \cite{GlHo08}. 
By means of real analytical local charts (see e.g. the analytic inverse theorem in \cite{92KRAPRA}), $\gamma(t)$ is locally mapped to the plan $\{ x_3=0 \}$ by some $\phi$ with $\phi(\intj(\gamma(t)))\subset \{ x_3>0 \}$. Moreover, we can require that 
\begin{equation*}
d \phi (\nu) \text{ is normal to } x_{3}=0 \text{ on } \{x_{3} =0 \}.
\end{equation*}
To obtain this property, consider
\begin{equation*}
\hat{\phi} : (x_{1},x_{2},x_{3}) \mapsto (x_{1},x_{2},0) + x_{3} d \phi_{\phi^{-1}(x_{1},x_{2},0)} (\nu). 
\end{equation*}
Then $\hat{\phi}$ is analytic and invertible (locally) in a neighborhood of $\{ x_{3} = 0\}$. Now, replace $\phi$ by $\hat{\phi}^{-1} \circ \phi$ to obtain the requirement. \par
Now call
\begin{equation*}
g:=\partial_{x_3}(\psi\circ\phi^{-1});
\end{equation*}
it satisfies 
\begin{equation*}
a \cdot \nabla^2g + b \cdot \nabla g + cg=0,
\end{equation*}
with a real-analytic Dirichlet boundary condition given on $x_3=0$ and analytic coefficients $a$, $b$ and $c$.
We use the following result of Cauchy-Kowalevsky type (see e.g. Morrey \cite[Theorem 5.7.1']{Mo08})
\begin{theorem}\label{th:cauchy-kow}
Let $f$, $a$, $b$ and $c$ be real analytic functions on
$$
\overline{G_R}:=\overline{B}_{\R^N}(0,R) \cap \{(x_1,...,x_N),\, x_N\geq 0\},
$$
and $y\in H^2(G_R)$  satisfying 
\begin{equation*}
a \cdot \nabla^2y + b \cdot \nabla y + c=f \ \text{ in } \ G_R
\ \text{ and } \ 
y=0 \ \text{ on } \ x_N=0.
\end{equation*}
Assume that for some constant $A$ and $L$ we have 
\begin{equation} \label{EstAnalytique}
|\nabla^p a(x)|, \ |\nabla^p b(x)|, \ |\nabla^p c(x)|, \ |\nabla^p f(x)| \leq L A^{|p|},
\end{equation}
for any multi-index $p$. There exists then $R'<R$ depending only on $N$, $A$, $L$ and $R$ such that $y$ may be extended analytically on $B_{\mathbb R^N}(0,r)$ for any $r<R'$.
\end{theorem}
Thus for each $x \in \gamma(t)$, $\nabla_x\psi \cdot \nu$ can be analytically extended on a neighborhood $U_x$ of $x$ across $\gamma(t)$. Integrating $\nabla\psi \cdot \nu$ along $\nu$ we deduce that $\psi$ is real analytic and by unique continuation its extension is also harmonic on $U_x$. Moreover, using the continuity of $\gamma$ in the variable $t$ from $[0,1]$ to $C^{\omega}(\S^{2})$\hskip1pt\footnote{We refer to \cite{92KRAPRA} for a definition of the topology on this space.} (so that the coefficients satisfy \eqref{EstAnalytique} uniformly in time), we can find $U_{x}$ so that $\psi$ is analytically extended on in $U_x$ for each $s$ in some neighborhood of $t$. Using the compactness of $\cup_t\{(t,\gamma(t))\}$ we see that we can choose $\eta>0$ uniform in $t$, such that for all $t$, $\psi$ can be analytically (and hence harmonically) extended in $V_{\eta}(\gamma(t))$. Lemma \ref{prop:etendrepsi} follows.
\end{proof}
\ \par
\begin{proof}[Back to the proof of Proposition \ref{ProStep1}]
We take $\gamma(t)=\phi^X(t,0,\gamma_0)$. Due to the regularity of $X$, $\gamma(t)$ is analytic for any $t$, and applying Lemma \ref{prop:etendrepsi} we deduce a function $\psi$. 
Reducing $\eta>0$ given by Lemma \ref{prop:etendrepsi} if necessary, we may assume that $V_{\eta}(\gamma(t))$ does not meet $\partial \Omega$. \par
By compactness of $[0,1]$, for a given $\eps>0$ we can choose $0\leq t_1<\cdots<t_N\leq 1$ and $\delta_1, \dots, \delta_N$ such that 
\begin{gather}
\nonumber
[0,1]\subset\cup_{i=1}^N (t_i-\delta_{i},t_i+\delta_{i}), \\
\nonumber
\forall t \in [t_i-\delta_i,t_i+\delta_i], \ \ \gamma(t) \subset V_{\eta/2}(\gamma(t_i)), \\
\label{eq:continuiteunifX}
\forall s, t \in [t_i-\delta_i,t_i+\delta_i], \ \ \| \psi(s,\cdot) - \psi(t,\cdot) \|_{C^k(\overline{V_{\eta}}(\gamma(t_i)))} \leq \eps.
\end{gather}
For each $i=1, \dots, N$, we consider
\begin{gather*}
K_i:=V_{\eta/2} [\intj(\gamma(t_i))] \cup V_{\eta/2}(\partial \Omega \setminus \Gamma), \\
\psi_{i}:= \psi(t_{i},\cdot).
\end{gather*}
Reducing $\eta$ again if necessary, we may assume that the connected component of $\Omega \setminus K_{i}$ in $\R^{3} \setminus K_{i}$ meets $\R^{3} \setminus \overline{\Omega}$. This is possible since the connected component of $\Omega \setminus [\intj(\gamma(t_i)) \cup (\partial \Omega \setminus \Gamma)]$ in $\R^{3} \setminus K_{i}$ does meet $\R^{3} \setminus \overline{\Omega}$, since $\Gamma \not = \emptyset$. It follows that each connected component of $\R^{3} \setminus K_{i}$ meets $\R^{3} \setminus \overline{\Omega}$. \par
Now we use the following harmonic approximation theorem (see \cite[Theorem 1.7]{GARD}):
\begin{theorem}\label{th:approxharm1}
Let ${\mathcal O}$ be an open set in $\R^{N}$ and let $K$ be a compact set in $\R^{N}$
such that that ${\mathcal O}^{*} \setminus K$ is connected, where ${\mathcal O}^{*}$ is the Alexandroff compactification of ${\mathcal O}$. Then, for each  function $u$ which is harmonic on an open set
containing $K$ and each $\varepsilon>0$, there is a harmonic function $v$ in ${\mathcal O}$
such that $\|v - u\|_{\infty} < \varepsilon$ on $K$.
\end{theorem}
Recall that the Alexandroff compactification of ${\mathcal O}$ is obtained by adding a point, say $\{ \infty \}$ to ${\mathcal O}$ and to consider the topology generated by the open sets of ${\mathcal O}$, and the sets of the form $\{ \infty \} \cup ({\mathcal O} \setminus K)$, with $K$ a subset of ${\mathcal O}$. \par
\begin{remark}
One may state the same result with the $C^{k}$ norm instead of the uniform one. It suffices to consider a compact $\tilde{K}$ whose interior contains ${K}$, apply the above result on $\tilde{K}$ and use standard properties of harmonic functions.
\end{remark}
We choose points $Y_1,...,Y_P$ in each connected component of $\R^{3} \setminus K_{i}$, outside $\overline{\Omega}$. We apply the preceding result with ${\mathcal O}= \R^{3} \setminus \{Y_{1},\dots, Y_{P} \}$ and $K=K_{i}$. In that case, ${\mathcal O}^{*}$ can be thought as the quotient of $\S^{3} = \R^{3} \cup \{ \infty \}$ by the identification of $Y_{1},\dots, Y_{P}$ and $\infty$. \par
Therefore, for any $\nu>0$, we get a map $\hat{\psi_i}$ in $C^{\infty}(\R^3 \setminus \{Y_1,...,Y_P\};\R)$ such that $\hat{\psi_i}$ is harmonic on $\R^3 \setminus \{Y_1,...,Y_P\}$ and such that 
\begin{gather}
\label{eq:psipsichapproche}
\| \hat{\psi_i}-\psi_i \|_{ C^{k+2}(\overline{V_{\eta/2}[\intj(\gamma(t_i))] }) } < \nu, \\
\label{eq:petitessesurbordcontrole}
\| \hat{\psi_i} \|_{C^{k+2}(\overline{V_{\eta/2}}(\partial \Omega\setminus\Gamma))}<\nu.
\end{gather}
Since $\hat{\psi_i}$ is harmonic in $\Omega$, there holds
\begin{equation}
\int_{\partial \Omega} \nabla \hat{\psi_i}\cdot n \, d\sigma=0.
\end{equation}
In order for \eqref{eq:euler.4} to be satisfied we consider ${d}_i$ in $C^{\infty}(\partial \Omega;\R)$ such that
\begin{gather}
\label{eq:5.1}
d_{i} = \nabla \hat{\psi}_i \cdot n \text{ on }\partial \Omega\setminus \Gamma, \\
\label{eq:5.2}
\| d_i \|_{C^{k+1}(\partial \Omega)} \leq C \| \nabla \hat{\psi_i} \cdot n \|_{C^{k+1}(\partial \Omega)}, \\
\label{eq:6}
\int_{\partial \Omega}d_{i} \, d\sigma=0.
\end{gather}
and introduce the harmonic function $h_{i}$ in $\overline{\Omega}$ by the Neumann problem
\begin{gather}
\label{Neumannhi1}
\Delta h_{i} =0 \text{ in } \Omega, \\
\label{Neumannhi2}
\frac{\partial h_{i}}{\partial n} =d_{i} \text{ in } \Omega, \\
\nonumber
\int_{\Omega} h_{i} =0.
\end{gather}
Note that in particular that by standard elliptic estimates,
\begin{equation} \label{eq:petitessefrakh}
\| {h}_i\|_{C^{k+1}(\bar{\Omega})} \leq C \nu.
\end{equation}
We introduce
\begin{equation} \label{DefCheckPsi}
\check{\psi}_i:=\hat{\psi}_i-{h}_i.
\end{equation}
Taking a partition of unity $\chi_i$  associated to the covering of $[0,1]$ by the intervals $(t_i-\delta_i,t_i+\delta_i)$, we define
\begin{equation}
\theta(t,x):=\sum_{n=1}^N \chi_i(t) \check{\psi}_i(x).
\end{equation}
Due to \eqref{eq:5.1} and \eqref{Neumannhi2}, $\theta$ satisfies \eqref{eq:4.9}.
Moreover according to \eqref{eq:continuiteunifX}, \eqref{eq:psipsichapproche}, \eqref{eq:petitessesurbordcontrole} and \eqref{eq:petitessefrakh} we have for $\nu$ small enough with respect to $\eps$ and for some $C>0$
\begin{equation}
\label{eq:thetaetpsiproches}
\sup_{t\in[0,1]}\| \nabla \theta -\nabla \psi\|_{C^k(\overline{V_{\eta/3}}[\gamma(t)])} \leq C \eps.
\end{equation}
In particular by supposing $\eps$ small enough, we have a uniform estimate
\begin{equation} \label{eq:bornesurtheta}
\| \nabla \theta \|_{C^k({\phi^{\nabla \theta}(0,t,\gamma_0)})}
\leq \| \nabla \psi \|_{C^k({\phi^{\nabla \psi}(0,t,\gamma_0)})}+1.
\end{equation}
As long as $\phi^{\nabla \theta}(t,0,\gamma_0)$ remains in $V_{\eta/3}(\gamma(t))$ one has, using Gronwall's lemma,
\begin{multline*}
\| \phi^{\nabla \theta}(t,0,\gamma_0) - \phi^{\nabla \phi}(t,0,\gamma_0) \|_\infty  \\
\leq \| \nabla \theta(t,\cdot) - \nabla \phi(t,\cdot) \|_{C^0([0,1];C^0(\overline{V_{\eta/3}}[\gamma(t)])}
\exp(\| \nabla \psi \|_{L^1(0,1;{\mathcal{L}ip(V_{\eta/3}[\gamma(t)])}}).
\end{multline*}
Then by reducing $\eps$ if necessary, we get thanks to \eqref{eq:thetaetpsiproches} that this is valid for all time in $[0,1]$. Differentiating $\phi$ with respect to $x$ up to the order $k$, we obtain in the same way that for all $t \in [0,1]$,
\begin{multline*}
\| \phi^{\nabla \theta}(t,0,\gamma_0) - \phi^{\nabla \phi}(t,0,\gamma_0) \|_{C^k([0,1])} \\
\leq \| \nabla \theta - \nabla \phi \|_{C^0([0,1];C^k(\overline{V_{\gamma/3}(\gamma(t))})}
\exp(C\| \nabla \psi \|_{L^1((0,1);W^{k+1,\infty}(V_{\eta/3}[\gamma(t)]))}).
\end{multline*}
This ends the proof Proposition \ref{ProStep1}.
\end{proof}
\ \par
%
%
%
%
%
\subsubsection{Case of a smooth $2$-sphere moved by a special analytic isotopy}
In this paragraph, Proposition \ref{ProStep1} is extended to the following.
\begin{proposition}\label{ProStep2}
The conclusions of Proposition \ref{th:3} are satisfied if we assume that $\gamma_0$ is a smooth $2$-sphere ({\it i.e.} $\gamma_{0}$ is the image of $\S^{2}$ by $f_0$ a $C^{\infty}$ embedding) and that $X$ is moreover in $C^{0}([0,1]; C^\omega(\Omega;\R^3))$.
\end{proposition}
\begin{proof}
Due to a result of H. Whitney (see \cite{Whi36}), $\gamma_{0}$ is imbedded in a smooth family of surfaces $\gamma_{\nu}$, $\nu \in (-\nu_{0},\nu_{0})$, with $\gamma_{0}=\gamma_{\nu}$ for $\nu=0$, $\gamma_{\nu} \cap \gamma_{\nu'} = \emptyset$ for $\nu \neq \nu'$ and $\gamma_{\nu}$ real analytic for $\nu \neq 0$. 
It follows that either for $\nu>0$ or for $\nu<0$, one has
\begin{equation} \label{gammanu}
\gamma_{0} \subset \intj[\gamma_{\nu}].
\end{equation}
Without loss of generality, we assume that \eqref{gammanu} holds for $\nu>0$. The family $\gamma_{\nu}$ being smooth with respect its parameters, one has 
\begin{equation} \label{CVGammaNu}
{\gamma}_{\nu} \ \to \ \gamma_0 \ \text{ in } \ C^\infty(\S^2) \text{ as } \nu \rightarrow 0^{+}.
\end{equation}
We can then apply Proposition \ref{ProStep1} to $X$ on ${\gamma}_{\nu}$ for $\nu>0$ small. We construct a $\theta_\eps$ such that \eqref{eq:thetaharm}, \eqref{eq:4.9}, \eqref{eq:thetarestedansomega} and \eqref{eq:thetafaitletravail} apply for ${\gamma}_{\nu}$ instead of $\gamma_0$.
The construction also generates a family $\psi$ satisfying \eqref{SysNeumann} in $\intj [\phi^{X}(t,0,\gamma_{\nu})]$. \par
Now we have uniform bounds on the $\| \nabla \psi\|_{C^{k}(\phi^{\nabla\psi}(t,0,\gamma_{\nu}))}$ with respect to $\nu$, because the constants of the elliptic estimates (see e. g. \cite[Theorem 6.30 and Lemma 6.5]{gilbtrud}) are uniform with respect to $\nu$ thanks to \eqref{CVGammaNu}. We deduce from \eqref{eq:bornesurtheta} that we have uniform bounds
on $\| \nabla \theta_\eps \|_{C^k(\phi^{\nabla \theta_\eps}(t,0,\gamma_{\nu}))}$ as $\nu\to 0^{+}$. 
By Gronwall's lemma one gets
\begin{multline*}
\| \phi^{\nabla\theta_\eps}(t,0,\tilde{\gamma_0}) - \phi^{\nabla \theta_\eps}(t,0,\gamma_0) \|_{C^k(\S^2)} \leq \\
C_{k} \| \gamma_0 - \gamma_{\nu}\|_{C^k(\S^2)}
\exp \left( \int_0^1 \| \nabla \theta_\eps \|_{C^{k+1}(\intj (\phi^{\nabla \theta_\eps}(t,0,\gamma_{\nu}))} \, dt \right).
\end{multline*}
Hence we deduce the claim by taking $\nu$ small enough.
Of course by reparameterizing $X$ and thus $\theta$ in time we can always assume that $\theta$ is compactly supported in time.
\end{proof}
\subsubsection{Case of smooth embedded contractible two-sphere moved by a smooth isotopy}
We now prove Proposition \ref{th:3}. 
\begin{proof}[Proof of Proposition \ref{th:3}]
In the general case, we assume that $X$ is merely $C^{\infty}(\overline{\Omega};\R^{3})$. Let $\lambda>0$ such that
\begin{equation*}
\max_{t \in [0,1]} \mbox{dist}(\phi^{X}(t,0,\gamma_{0}) , \partial \Omega) >2 \lambda.
\end{equation*}
We define ${\mathcal U}_{t}:= {\mathcal V}_{\lambda}(\intj(\gamma(t)))$. Reducing $\lambda$ if necessary, we can obtain that for all $t$, ${\mathcal U}_{t}$ is diffeomorphic to a ball. \par
Now we can use Whitney's approximation theorem (see e.g. \cite[Proposition 3.3.9]{92KRAPRA}),
for any $\mu>0$ and any $k \in \N$ there exists $X_{\mu} \in C([0,1];C^\omega(\R^3))$ such that
\begin{equation*}
\| X_{\mu} - X \|_{C([0,1];C^{k+1}({\mathcal U}_{t}))} \leq \mu.
\end{equation*}
Moreover, we can ask that
\begin{equation*}
\div X_{\mu} =0 \ \text{ in } \ [0,1] \times \R^{3}.
\end{equation*}
To see this, we use the fact that ${\mathcal U}_{t}$ is a topological ball; hence its second de Rham cohomology space is trivial, and any divergence-free vector field (in particular $X$) on ${\mathcal U}_{t}$ is of the form $\curl A$.
Hence for each time $t$ one can apply Whitney's approximation theorem (at order $k+2$) on $\varphi A$ where $X = \curl A$ and $\varphi(t,x)$ is a smooth cutoff function equal to $1$ on ${\mathcal U}_{t}$ and to $0$ for $d(x, {\mathcal U}_{t}) \geq 2 \lambda$.  We obtain an approximating vector field $B_{\mu}$ and define $X_{\mu} := \curl B_{\mu}$. Using as before the compactness of the time interval $[0,1]$ and a partition of unity (as for the proof of Proposition \ref{ProStep1}), we can obtain a smooth approximation uniformly in time. \par 
Now we apply Proposition \ref{ProStep2} with $\tilde{X}_{\mu}$, apply Gronwall's lemma and get the result for $\mu$ small enough. 
\end{proof}
%
%
%
%
%
%
%
%
%
\section{Proof of Theorem \ref{th:fond.1}}
\label{Sec:Proof}
This section is devoted to the proof of Theorem \ref{th:fond.1}. As in \cite{96COR-2,00Gla,GlHo08}, we start with the case where $u_{0}$ is small, and then treat the general case.
\subsection{Preliminaries}
First we introduce some functions on $\Omega$ in order to take its topology into account, more precisely to describe its first de Rham cohomology space.
We recall the following construction. \par
Let $\Sigma_1,...,\Sigma_g$ be $g$ smooth manifolds  with boundaries of dimension $n-1$ inside $\overline{\Omega}$ such that:
\begin{itemize}
\item for all $i$ in $\{1, \dots, g \}$, $\partial \Sigma_i\subset \partial \Omega$ and $\Sigma_{i}$ is transverse to $\partial \Omega$,
\item for all $i,j$ in $\{1, \dots, g \}^2$ with $i\neq j$,  $\Sigma_i$ and $\Sigma_j$ intersect transversally (which, by definition, includes the case of an empty intersection),
\item $\Omega \setminus \cup_{i=1}^g \Sigma_i$ is simply connected.
\end{itemize}
For $i=1, \dots, g$, we consider
\begin{equation*}
{X}_i:=\{p\in H^1(\Omega\setminus \cup_{k=1}^g \Sigma_k),\ 
[p]_i=\text{constant},\ [p]_j=0,\ j\neq i \},
\end{equation*}
where $[p]_k=p_{|\Sigma_k^+}-p_{|\Sigma_k^-}$ is the jump of $p$ on each arbitrarily fixed side of $\Sigma_k$ in $\Omega$. Then by Lax-Milgram's Theorem there exists a unique $q_i \in X_{i}$ such that for all $p\in {X}_i$ 
\begin{equation*}
\int_{\Omega}\nabla q_i \cdot \nabla p \, dx=[p]_i,
\end{equation*}
which leads to the existence of a unique $p_i$ such that 
\begin{gather*}
\Delta p_i = 0 \ \text{ in } \ \Omega\setminus \cup_{k=1}^g \Sigma_k, \\
\partial_n p_i =0 \ \text{ on } \ \partial \Omega, \\
{[p_i]}_{i} = 1 \ \text{ and } \ {[p_i ]}_j = 0 \ \text{ for } \ j\neq i, \\
{[\partial_n q_i]}_i=0,
\end{gather*}
and we take
\begin{equation*}
Q_i:=\nabla p_i,
\end{equation*}
which is regular in $\overline{\Omega}$. Then we have the following result.
\begin{proposition}[see e.g. \cite{temam79}, Appendix I, Proposition 1.1] \label{th:tem.1}
For any $X\in L^{2}(\Omega;\R^{3})$ such that 
$\curl X=0$ in $\Omega$, there exist $\chi$ in $H^1(\Omega;\R)$ and $\alpha_{1},\dots,\alpha_{g}$ in $\R$ such that
\begin{equation}
X=\nabla \chi + \sum_{i=1}^g \alpha_i {Q_i}.
\end{equation}
\end{proposition}
\subsection{A fixed point operator}
Now we introduce an operator, whose fixed point will give a solution of our problem when taking $u_{0}$ (suitably small) into account. \par
%
%
%
We introduce $R>0$ such that 
\begin{equation*}
\overline{\Omega} \subset B_{R}:=B_{\R^{3}}(0,R).
\end{equation*}
We introduce a linear continuous extension operator $\pi: C(\overline{\Omega};\R^3) \to C_0(B_R;\R^3)$ such that $\forall k\in \N$ and all $\beta \in (0,1)$, $\pi$ is continuous from $C^{k,\beta}(\overline{\Omega};\R^3)$ to $C_0^{k,\beta}(B_R;\R^3)$. \par
Let $\delta \in (0,1)$, and consider $\mu \in C^{\infty}_{0}([0,+\infty);\R)$ with support in $[0,\delta]$ and with value $1$ on a neighborhood of $0$. \par
\ \par
Given $\eps>0$ we denote 
\begin{equation*}
\bar{y}:=\nabla \theta,
\end{equation*}
the potential flow obtained by Proposition \ref{th:3} with $X$ obtained from Theorem \ref{th:isotop.1}.
For some $\nu \in (0,1)$ which will be small in the sequel we define with $\alpha \in (0,1)$:
\begin{equation*}
X_\nu := \Big\{ u \in L^{\infty}((0,1);C^{k,\alpha}({\Omega};\R^3)), \ \div u=0 \text{ in } (0,1) \times \Omega, \ 
\| u-\bar{y} \|_{L^{\infty}(0,1; C^{k,\alpha}({\Omega}))} \leq \nu \Big\}.
\end{equation*}
%
It is straightforward to check that $X_{\nu}$ is a closed convex subset of $L^{\infty}((0,1) \times \Omega)$. \par
\ \par
We extend $\bar{y}$ by $\pi$ and still denote $\bar{y}$ the extended function. Now, given $u \in X_{\nu}$, we associate $F(u)$ as follows. First, we introduce
\begin{equation*}
\tilde{u}:=\bar{y}+\pi(u-\bar{y})= \pi(u).
\end{equation*}
Next we consider $\omega_u \in L^{\infty}(0,1;C^{k-1,\alpha}(B_R;\R^{3}))$ as the solution of the following transport equation:
\begin{gather}
\label{eq:conditioninitialeomega}
\omega_u (\cdot,0)=\curl \pi(u_0) \ \text{ in } \ {B}(0,R), \\
\label{eq:transportomegaparutilde}
\partial_t\omega_u + (\tilde{u} \cdot \nabla) \omega_u = (\omega_u \cdot \nabla)\tilde{u}
- (\div \tilde{u}) \omega_{u} \ \text{ in } \ (0,1)\times B_R.
\end{gather}
%
%
Due the support of $\tilde{u}$, one sees by using characteristics that the system \eqref{eq:conditioninitialeomega}-\eqref{eq:transportomegaparutilde} is well-posed and that indeed $\omega_{u}$ has the claimed regularity (details for obtaining the regularity can be found in \cite{00Gla}).  \par
Now observing that
\begin{equation*}
\curl (a \wedge b) =  (\div b) a - (\div a) b + (b \cdot \nabla) a - (a \cdot \nabla) b,
\end{equation*}
we easily deduce that
\begin{equation*}
\div(\omega_{u})=0 \ \text{ in } \ B_R, 
\end{equation*}
so that $\omega_{u}$ can be written in the form $\curl \hat{v}$ in $B_R$.
Hence it is classical (since $\int_{\partial \Omega}u_0 \cdot n \, d\sigma=\int_{\partial \Omega}\bar{y} \cdot n \, d\sigma=0$) that there exists a unique ${v}\in L^{\infty}(0,1;C^{k,\alpha}(\overline{\Omega};\R^3))$ such that 
\begin{subequations}
\label{eq:defmathfrav}
\begin{gather}
\label{eq:defmathfrav.1}
\curl {v} = \omega_u \ \text{ in } \ [0,1] \times \Omega, \\
\label{eq:defmathfrav.2}
\div {v} = 0 \ \text{ in } \ [0,1] \times \Omega,\\
\label{eq:defmathfrav.3}
{v} \cdot n = \mu(t)u_0 \cdot n +\bar{y} \cdot n \ \text{ on } \ [0,1] \times \partial \Omega, \\
\label{eq:defmathfrav.4}
\int_\Omega {v} \cdot Q_i \, dx = 0, \  i=1,\dots,g.
\end{gather}
\end{subequations}
According to Proposition \ref{th:tem.1}, we can determine $g$ time-dependent functions $\lambda_1,\dots,\lambda_g$ such that, if we define 
\begin{equation} \label{eq:defdev}
V := {v} + \sum_{i=1}^g \lambda_i(t) {Q}_i,
\end{equation}
we have for all $j = 1, \dots, g$
\begin{equation}
\label{eq:lambdai1}
\int_\Omega V(0) \cdot Q_j \, dx = \int_\Omega u_0 \cdot Q_j \, dx=0,\\
\end{equation}
and for all $t \in [0,1]$,
\begin{equation} 
\label{eq:lambdai2}
\int_\Omega V(t,x) \cdot Q_{j}(x) \, dx - \int_\Omega u_{0}(x) \cdot Q_{j}(x) \, dx   = - \int_{0}^{t} \int_{\Omega} (u(\tau,x) \wedge \omega_u(\tau,x)) \cdot Q_j(x) \, dx \, d \tau.
\end{equation}
This is possible in a unique way since the matrix $(\int_\Omega Q_i \cdot Q_j \, dx)_{i,j}$ is invertible, as a Gram matrix of independent functions. \par
Now we finally define
\begin{equation*}
F(u):=V.
\end{equation*}
\ \par
\subsection{Finding a fixed point}
Our goal is to prove hereafter:
\begin{proposition} \label{th:Faunpointfixe}
Given $\nu>0$, if $\| u_0 \|_{C^{j,\alpha}(\Omega)}$ is small enough, $F$ admits a fixed point in $X_\nu$.
\end{proposition}
\noindent
We will use the following lemma (see e.g. \cite[Theorem 3.14]{BCD}):
\begin{lemma}\label{EstTransport}
Let $j \in \N$, $a \in (0,1)$. Let $f$, $v$, and $g$ be elements of $L^{\infty}(0,1;C^{j,\alpha}(B_{R};\R^3))$ satisfying
\begin{equation*}
\partial_t f + (v \cdot \nabla) f = g,
\end{equation*}
with $v$ and $f(0,\cdot)$ compactly supported in $B_R$. Then for some $C>0$ depending on $j$ and $\alpha$ only,
there holds 
\begin{multline*}
\| f(t,\cdot) \|_{C^{j,\alpha}(B_{R})} \leq
\exp \left( C\int_{0}^{t} V(s) \, ds \right) \\
\left[\| f(0,\cdot) \|_{C^{j,\alpha}(B_{R})} 
+ \int_{0}^{t} \exp \left( - C\int_{0}^{\tau} V(s) \, ds \right)  \| g (\tau,\cdot) \|_{C^{j,\alpha}(B_{R})} \, d \tau\right],
\end{multline*}
with
\begin{equation*}
V(s):=  \| \nabla v (s,\cdot) \|_{C^{j-1,\alpha}(B_{R})} \ \text{ if } \ j \geq 1 \ \text{ and } \
V(s):= \| \nabla v (s,\cdot) \|_{L^{\infty}(B_{R})} \ \text{ if } \ j =0. 
\end{equation*}
\end{lemma}
\noindent
\begin{proof}[Proof of Proposition \ref{th:Faunpointfixe}]
We establish Proposition \ref{th:Faunpointfixe} and find a fixed point of $F$ in $X_{\nu}$ via Schauder's fixed point theorem. Accordingly, we prove that, $\nu$ being fixed and for $\| u_0 \|_{C^{k,\alpha}}$ small enough, $F$ sends $X_{\nu}$ into itself, that $F$ is continuous and $F(X_{\nu})$ is relatively compact for the uniform topology on $X_{\nu}$. \par
\ \par
\noindent
$\bullet$ Using Lemma \ref{EstTransport}, we see that
\begin{multline*}
\| \omega_{u}(t,\cdot) \|_{C^{k-1,\alpha}(B_{R})} \leq
\exp \left( C\int_{0}^{t} V(s) \, ds \right) \bigg[\| \omega_{u}(0,\cdot) \|_{C^{k-1,\alpha}(B_{R})} \\
+ C \int_{0}^{t} \exp \left( - C\int_{0}^{\tau} V(s) \, ds \right)  \| \tilde{u} (\tau,\cdot) \|_{C^{k,\alpha}(B_{R})} \| \omega_{u} (\tau,\cdot) \|_{C^{k-1,\alpha}(B_{R})} \, d \tau \bigg],
\end{multline*}
with as before
\begin{equation*}
V(s):=  \| \nabla \tilde{u} (s,\cdot) \|_{C^{k-2,\alpha}(B_{R})} \ \text{ if } \ k \geq 2 \ \text{ and } \
V(s):= \| \nabla \tilde{u} (s,\cdot) \|_{L^{\infty}(B_{R})} \ \text{ if } \ k =1. 
\end{equation*}
We apply Gronwall's lemma to
\begin{equation*}
t \mapsto \| \omega_{u}(t,\cdot) \|_{C^{k-1,\alpha}(B_{R})}  \exp \left( - C\int_{0}^{t} V(s) \, ds  \right),
\end{equation*}
and deduce
\begin{equation*}
\| \omega_u(t)\|_{C^{k-1,\alpha}(B_{R})} \leq \| \omega_u(0)\|_{C^{k-1,\alpha}(B_{R})}
e^{C\| \tilde{u}\|_{L^{\infty}(0,1;C^{k,\alpha}(B_{R}))}}.
\end{equation*}
Thus with the definition of $X_{\nu}$ and the continuity of $\pi$, we obtain that
\begin{equation}
\label{eq:estimeesurOmegadeomega}
\| \omega_u(t)\|_{C^{k-1,\alpha}(B_{R})} \leq \| \omega_u(0)\|_{C^{k-1,\alpha}(B_{R})}
e^{C (\| \overline{y}\|_{L^{\infty}(0,1;C^{k,\alpha}({\Omega}))} +1)}.
\end{equation}
Now by \eqref{eq:lambdai2} we have
\begin{equation*}
|\lambda_i(t)| \leq C (| \lambda_{i}(0)| + t \| \omega_u \|_{C^{0,\alpha}(B_{R})} \| u \|_{C^{1,\alpha}(B_{R})} ),
\end{equation*}
and thus, with \eqref{eq:lambdai1}, \eqref{eq:estimeesurOmegadeomega} and the definition of $X_{\mu}$, we deduce
\begin{equation}
\label{eq:estimeesurlambdai}
|\lambda_i(t)| \leq C \| u_0\|_{C^{k,\alpha}(B_{R})} 
\left( 1+ 
t e^{ C (\| \overline{y} \|_{L^{\infty}(0,1;C^{k,\alpha}(\Omega))} + 1 ) } [\| \overline{y} \|_{L^{\infty}(0,1;C^{k,\alpha}({\Omega}))} + \nu] \right).
\end{equation}
Thus, by combining \eqref{eq:estimeesurOmegadeomega}, \eqref{eq:estimeesurlambdai}, and the elliptic estimates given by \eqref{eq:defmathfrav}-\eqref{eq:defdev}, we infer that
\begin{equation}
\| F(u) - \bar{y} \|_{L^{\infty}(0,1;C^{k,\alpha}(\overline{\Omega})) } \leq C(\| \overline{y} \|_{k,\alpha})
\| u_0 \|_{C^{k,\alpha}(\Omega)},
\end{equation}
for some constant $C$ depending on $k$, $\alpha$ and $\overline{y}$. It follows that for $\| u_{0}\|_{C^{k,\alpha}(\Omega)}$ small enough, $F$ sends $X_{\nu}$ into itself. \par
\ \par
\noindent
$\bullet$ That $F(X_{\nu})$ is relatively compact is seen easily: given $(u_{n}) \in X_{\nu}^{\N}$, the sequence $(F(u_{n}))$ belongs to $X_{\nu}$ and, following the construction, it is easy to see that $(\partial_{t} F(u_{n}))$ is bounded in $L^{\infty}(0,1;C^{k-1}(\Omega))$. The conclusion follows then from Ascoli's theorem. \par
\ \par
\noindent
$\bullet$ Finally, that $F$ is continuous for the uniform topology can be seen as follows. Let us be given $(u_{n}) \in X_{\nu}^{\N}$ converging uniformly to $u \in X_{\nu}$. The flows $\Phi_{n}$ associated to $\tilde{u}_{n}$ converge uniformly towards the flow $\Phi$ associated to $\tilde{u}$. Hence one can see that $\omega_{u_{n}}$ converges uniformly to $\omega_{u}$. Due to the bounds on $\omega_{u}$, this convergence also takes place in $L^{\infty}(0,1;C^{k-1,\beta}(B_R))$ for any $\beta < \alpha$. We deduce in a straightforward manner the convergence of $(v_{n})$ and $(\lambda_{i}^{n})$ corresponding to $u_{n}$ towards $v$ and $\lambda_{i}$ corresponding to $u$, and the conclusion follows. \par
\ \par
It follows that $F$ admits a fixed point $u$ in $X_{\nu}$.
This concludes the proof of Proposition~\ref{th:Faunpointfixe}.
\end{proof}
\subsection{Relevance of the fixed point}
Call $\overline{u}$ the fixed point obtained above. \par
\ \par
\noindent
{\bf 1.} Let us first check that $\overline{u}$ is a solution of the Euler equation in $[0,1] \times \Omega$. From \eqref{eq:transportomegaparutilde}, and since $\tilde{u}=u$ in $[0,1] \times \Omega$, we deduce that
\begin{equation*}
\curl (\partial_{t} u + (u \cdot \nabla) u )=0 \ \text{ in } \ [0,1] \times \Omega.
\end{equation*}
From \eqref{eq:lambdai2} and $(u \cdot \nabla) u = \nabla \frac{|u|^{2}}{2} + (\curl u) \wedge u$, we see that
\begin{equation*}
\int_{\Omega} (\partial_{t} u + (u \cdot \nabla) u ) \cdot Q_{i} \, dx =0,
\end{equation*}
which together with Proposition \ref{th:tem.1} proves the claim. \par
\ \par
\noindent
{\bf 2.} Now we prove that $\bar{u}$ fulfills the requirements of Theorem \ref{th:fond.1} for $\| u_{0} \|_{C^{k,\alpha}}$ small enough. \par
Given $x \in B_R$, we have
\begin{equation*}
\dot{\phi}^{\bar{u}}(t,0,x) - \dot{\phi}^{\bar{y}} (t,0,x)=\bar{u}(t,{\phi^{\bar{u}}}(t,0,x))-\bar{y}(t,{\phi^{\bar{u}}}(t,0,x))
+\bar{y}(t,{\phi^{\bar{u}}}(t,0,x)) - \bar{y}(t,{\phi^{\bar{y}}}(t,0,x)).
\end{equation*}
We deduce easily that for a constant $C>0$ depending on $\overline{y}$ only,
\begin{equation*}
|\dot{\phi}^{\bar{u}}(t,0,x) - \dot{\phi}^{\bar{y}}(t,0,x)| \leq \nu + C | {\phi^{\bar{u}}}(t,0,x) - {\phi^{\bar{y}}}(t,0,x)| .
\end{equation*}
Thus by Gronwall's lemma we have $|x(t)-y(t)| \leq C\nu $ where $C$ depends only on $\bar{y}$. 
Reasoning in the same way for the derivatives (up to order $k$) with respect to $x$ of the flows, we obtain
\begin{equation*}
\| \phi^{\tilde{u}}(t, 0, \cdot ) - \phi^{\overline{y}}(t, 0, \cdot ) \|_{C^{k}(B_R)} \leq C \nu.
\end{equation*}
Hence taking $\nu$ small enough (and hence $\| u_0\|_{C^{k,\alpha}}$ even smaller), we can obtain \eqref{eq:exact2} and \eqref{eq:approx} for $T=1$. 
In order words, there exists $\overline{c}>0$, such that for any $u_{0}$ in $C^{k,\alpha}$ with $\| u_0\|_{C^{k,\alpha}} \leq c$, one can find a solution of the Euler equation for $T=1$, satisfying \eqref{eq:exact2} and \eqref{eq:approx}. \par
\ \par
\noindent
{\bf 3.} Let us now explain how we can obtain the result without the condition of smallness of $u_0$ (but for $T$ small enough). Given $u_{0} \in C^{k,\alpha}(\Omega)$, we rescale it by considering 
\begin{equation*}
v_0=\rho u_0,
\end{equation*}
with $\rho>0$ small enough so that $v_{0}$ satisfies $\| v_0\|_{C^{k,\alpha}} \leq \overline{c}$ .
Applying the above construction to $v_0$ gives us a solution $(u,p)$ of \eqref{eq:euler.1}-\eqref{eq:euler.3} defined on $t \in [0,1]$ with $u(0,\cdot)=v_0$ such that 
\begin{equation*}
\| \phi^u(1,0,\gamma_0) - \gamma_1 \|_\infty < \eps.
\end{equation*}
If we define $u_\rho$ by 
\begin{equation*}
u_\rho(t,x)=\frac{1}{\rho} u \left( \frac{t}{\rho}, x \right),
\end{equation*}
then $u_\rho$ is defined on $t\in [0,\rho]$ and 
\begin{equation*}
\| \phi^{u_\rho}(\rho,0,\gamma_0)-\gamma_1\| <\eps.
\end{equation*}
This concludes the proof of Theorem \ref{th:fond.1}. \par
%
%
%


\end{document}